\documentclass{amsart}
\usepackage{amsmath}
\usepackage[dvips]{graphicx}
\usepackage[usenames,dvipsnames]{xcolor}
\usepackage{colortbl}
\usepackage{multirow}
\usepackage{pgf,tikz}
\usepackage{subfig}
\usepackage{amsmath, amsthm, amscd, amsfonts, amssymb, color}
\usepackage[bookmarksnumbered, plainpages]{hyperref}
\usepackage{relsize}
\usepackage{longtable}
\newtheorem{thm}{Theorem}[section]
\newtheorem{cor}[thm]{Corollary}
\newtheorem{con}[thm]{Conjecture}
\newtheorem{lem}[thm]{Lemma}
\newtheorem{prop}[thm]{Proposition}
\newtheorem{defn}[thm]{Definition}
\newtheorem{rem}[thm]{Remark}

\numberwithin{equation}{section}
\begin{document}
\title[Permutation Codes With Minimum Kendall $\tau$-distance of Three]{New Bounds on the Size of  Permutation Codes With Minimum Kendall $\tau$-distance of Three}
\author[Abdollahi]{A. Abdollahi}%
\address{Department of Pure Mathematics, Faculty of Mathematics and Statistics,  University of Isfahan, Isfahan 81746-73441, Iran.}%
\address{School of Mathematics, Institute for Research in Fundamental Sciences (IPM), 19395-5746 Tehran, Iran.}
\email{a.abdollahi@math.ui.ac.ir}%
\author[Bagherian]{J. Bagherian}
\address{Department of Pure Mathematics, Faculty of Mathematics and Statistics,  University of Isfahan, Isfahan 81746-73441, Iran.}%
\email{bagherian@sci.ui.ac.ir}
\author[Jafari]{F. Jafari}
\address{Department of Pure Mathematics, Faculty of Mathematics and Statistics,  University of Isfahan, Isfahan 81746-73441, Iran.}%
\email{math\_fateme@yahoo.com}
\author[Khatami]{M. Khatami}%
\address{Department of Pure Mathematics, Faculty of Mathematics and Statistics,  University of Isfahan, Isfahan 81746-73441, Iran.}%
\email{m.khatami@sci.ui.ac.ir}
\author[Parvaresh]{F. Parvaresh}
\address{Department of Electrical Engeenering,  University of Isfahan, Isfahan 81746-73441, Iran.}%
\address{School of Mathematics, Institute for Research in Fundamental Sciences (IPM), 19395-5746 Tehran, Iran.}
\email{f.parvaresh@eng.ui.ac.ir}
\author[Sobhani]{R. Sobhani}%
\address{Department of Applied Mathematics and Computer Science, Faculty of Mathematics and Statistics,  University of Isfahan, Isfahan 81746-73441, Iran.}%
\email{r.sobhani@sci.ui.ac.ir}
\thanks{Corresponding Author: A. Abdollahi}
\thanks{F. Parvaresh is supported by in part by grant No. 1401680050.}
\subjclass[2010]{94B25;  94B65; 68P30}
\keywords{Rank modulation,  Kendall $\tau$-distance, permutation codes.}

\begin{abstract}
We study  $P(n,3)$, the size of the largest
subset of the set of all permutations $S_n$ with minimum Kendall
$\tau$-distance $3$. Using a combination of group theory and integer programming, we reduced the upper bound of $P(p,3)$ from $(p-1)!-1$ to $(p-1)!-\lceil\frac{p}{3}\rceil+2\leq (p-1)!-2$ for all primes  $p\geq 11$.  In special cases where $n$ is equal to $6,7,11,13,14,15$ and $17$ we reduced the upper bound of $P(n,3)$ by $3,3,9,11,1,1$ and $4$, respectively.
\end{abstract}
\maketitle
\section{Introduction}
Rank modulation was proposed as a solution to the
challenges posed by flash memory storages\cite{jiang}. In the rank
modulation framework, codes are permutation codes, where by a permutation code (PC) of length $n$ we simply mean a non-empty subset $\mathcal{C}$ of $S_n$, the set of all permutations of $[n]:=\{1,2,\ldots ,n\}$.
 Given a permutation $\pi:=[\pi(1),\pi(2),\ldots ,\pi(i),\pi(i+1),\ldots ,\pi(n)]\in S_n$,  an adjacent transposition, $(i, i + 1)$, for some $1\leq i\leq n-1$,  applied to $\pi$ will result in the permutation $[\pi(1),\pi(2),\ldots ,\pi(i+1),\pi(i),\ldots ,\pi(n)]$. For two permutations
$\rho,\pi\in S_n$, the Kendall $\tau$-distance between $\rho$ and $\pi$,
$d_K(\rho, \pi)$, is defined as the minimum number of adjacent
transpositions needed to transform  $\rho$ into $\pi$. Under the Kendall $\tau$-distance a PC of length $n$ with minimum distance $d$ can correct up to $ \frac{d-1}{2} $ errors
caused by charge-constrained errors \cite{jiang}.

 The maximum size of a PC of length $n$ and  minimum Kendall $\tau$-distance  $d$ is denoted by $P(n,d)$. Several researchers
have presented bounds on $P(n, d)$ (see \cite{barg,BE,jiang,V,WZYG,WWYF}), some of these results are shown in Table \ref{10101}. It is known that $P(n, 1)=n!$ and $P(n,2)=\frac{n!}{2}$. Also it is known that if $\frac{2}{3}\binom{n}{2}<d\leq \binom{n}{2} $, then $P(n,d)=2$ (see \cite[Theorem 10]{BE}). However, determining $P(n, d)$ turns out to be difficult for $3<d\leq \frac{2}{3}\binom{n}{2} $. In this paper, we study the upper bound of $P(n,3)$. By sphere packing bound (see \cite[Theorems 17 and 18]{jiang})  $P(n,3)\leq (n-1)!$. It is proved that if $n>4$ is a prime number or $4\leq n\leq 10$, then $P(n,3)\neq (n-1)!$ (see \cite[Corollary 2.5  and Theorem 2.6]{white} or \cite[Corollary 2]{BE}).

\begin{table}[!hbp]
	\begin{tabular}{|c|c|c|c|c|}
		\hline
		\cellcolor{yellow!50}{$n$}&\cellcolor{yellow!50}{6}&\cellcolor{yellow!50}{7} &\cellcolor{yellow!50}{11}&\cellcolor{yellow!50}{13}\\
		\hline
		LB & 102 \cite{V} & 588 \cite{BE} & $11!/20$ \cite{barg} & $13!/24$ \cite{barg} \\
		\hline
		UB& $5!-1^{a}$ & $6!-1^{a}$ & $10!-1^{a}$ & $12!-1^{a}$ \\
		\hline
		UB & \cellcolor{black!20!white}{$5!-4$}&\cellcolor{black!20!white}{$6!-4$}&\cellcolor{black!20!white}{$10!-10$}&\cellcolor{black!20!white}{$12!-12$}\\
		\hline
		\cellcolor{yellow!50}{$n$}&\cellcolor{yellow!50}{14}&\cellcolor{yellow!50}{15} &\cellcolor{yellow!50}{17}&\cellcolor{yellow!50}{prime  $n\geq 19$ }\\
		\hline
		LB& $2\times 12!$ \cite{jiang} &$15!/28$ \cite{barg} &$2\times 15!$ \cite{jiang} &$2\times(n-2)!$ \cite{jiang}\\
		\hline
		UB& $13!$ \cite{jiang}&$14!$ \cite{jiang}&$ 16!-1^{a} $ &$(n-1)!-1^{a}$ \\
		\hline
		UB&\cellcolor{black!20!white}{$13!-1$}&\cellcolor{black!20!white}{$14!-1$}&\cellcolor{black!20!white}{$16!-5$}&\cellcolor{black!20!white}{$(n-1)!-\lceil\frac{n}{3}\rceil +2$}\\
		\hline
	\end{tabular}
	\caption{  {\small Some results on  bounds of $P(n,3)$.  The superscripts shows the references from which the upper bound is taken, where ``a" is \cite{BE,white}, and  gray color shows our main results.}}\label{10101}
\end{table}

There are several works which uses optimization techniques to bound
the size of permutation codes under various distance metrics
(Hamming, Kendall, Ulam) (see \cite{Go,tar,V}).  In  Section II of this paper, we show that for any non-trivial subgroup of $S_n$, we can derive an integer programming  where the optimal
value of the objective function  gives an upper bound on $P(n,3)$. In Section III, by considering  Young  subgroups (see Definition \ref{young1}, below) of  $S_n$,  we can 
improve the upper bound of $P(n,3)$ as shown in Table \ref{10101}.

\section{Preliminaries}
Let $G$ be a finite group and denote by $\mathbb{C}[G]$ the complex group algebra of $G$. The elements of $\mathbb{C}[G]$ are of the formal sum 
\begin{equation}
\label{elementgroupring} 
\sum_{g \in G} a_g g,
\end{equation}
where $a_g\in \mathbb{C}$.  The complex group algebra is a $\mathbb{C}$-algebra with the following addition, multiplication and scaler product:
$$\sum_{g \in G} a_g g + \sum_{g \in G} b_g g=\sum_{g \in G} (a_g+b_g) \sigma, $$
$$\big(\sum_{g \in G} a_g g \big)  \big(\sum_{g \in G} b_g g\big)=\sum_{g \in G} \big(\sum_{g=g_1 g_2}a_{g_1} b_{g_2}\big) g,$$
$$\lambda \sum_{g \in G} a_g g =\sum_{g \in G}(\lambda a_g) g,$$
where $\lambda \in \mathbb{C}$ and $\sum_{g \in G} b_g g \in \mathbb{C}[G]$. 
If $a_g=0$ for some $g$, the term $a_g g$ will be neglected in \ref{elementgroupring} and $\sum_{g \in G} a_g g$ is written as
$a_1 g_1 +\cdots+a_k g_k$, where $\{g \;|\; a_g\neq 0\}=\{g_1,\dots,g_k\}$ is non-empty and otherwise $\sum_{g \in G} a_g g$ is denoted by $0$.\\ 
 For a non-empty subset $X$ of $G$, we denote by $\widehat{X}$ the element $\sum_{x\in X} x$ of $\mathbb{C}[G]$. 
 
Let $G$ be a finite group and $S$ be a non-empty inverse closed subset of $G$ not-containing the identity element $1$ of $G$. Consider the Cayley graph $\Gamma:=Cay(G,S)$ whose vertices are elements of $G$ and two vertices $g,h$ are adjacent if $gh^{-1} \in S$. Now we have a metric $d_\Gamma$ on $G$ defined by $\Gamma$ which is the length of a shortest path between two vertices in $Cay(G,S)$. For example if $G=S_n$ and $S=\{(1,2),(2,3),\dots,(n-1,n)\}$, the metric $d_\Gamma$ is the Kendall $\tau$-metric on $S_n$. Also if $G=S_n$ and $S=T \cup T^{-1}$, where $T:=\{ (a , a+1, \ldots , b) \;|\; a<b, a,b\in [n] \}$, the metric $d_\Gamma$ is the Ulam metric on $S_n$. 
\begin{defn}
For a positive ineteger $r$ and an element $g\in G$, the ball of radius $r$ in $G$ under the metric $d_\Gamma$ is denoted by $B_r^\Gamma(g)$ defined by 
$B_r^\Gamma(g)=\{h\in G \;|\; d_\Gamma(g,h)\leq r\}$.
\end{defn}
\begin{rem}\label{ballsize}
 Note that $B_r^{\Gamma}{(g)}=(S^r \cup\{1\})g$, where $S^r:=\{s_1\cdots s_r \;|\; s_1,\dots,s_r\in S\}$. Also note that since $S$ is inverse closed, $B_r^{\Gamma}{(g)}=S^r g$ for all $r\geq 2$. It follows that $|B_r^\Gamma(g)|=|B_r^\Gamma(1)|=|S^r \cup\{1\}|$ for all $g\in G$. 
\end{rem}
\begin{prop}\label{rel}
Let $G$ be a finite group and $d_{\Gamma}$ be the metric induced by the metric $Cay(G,S)$. Then a subset $C$  of  $G$ is a code with $min\{d_{\Gamma}(x,y)\,|\,x,y\in C\}\geq 3$ if and only if there exists $Y\subset G$ such that
	\begin{equation} \label{groupring}
	\widehat{(S \cup \{1\})} \widehat{C} =\widehat{G}-\widehat{Y},
	\end{equation}  
\end{prop}
\begin{proof}
Let $Y=G\setminus\cup_{c\in C}B_1^\Gamma(c)$. So $G=\cup_{c\in C}B_1^\Gamma(c)\cup Y$ and hence Remark \ref{ballsize} implies $\widehat{G}=\widehat{(S\cup\{1\}) C}+\widehat{Y}$.
	 Now the result follows from the fact that for any two distinct elements $c,c'$ in $C$, $(S\cup\{1\})c \cap (S\cup\{1\}) c'=\varnothing$ since otherwise $d_{\Gamma}(c,c')\leq 2$ that is a contradiction. This completes the proof. 
\end{proof}

Let $G$ be a finite group and $d_{\Gamma}$ be the metric induced by the metric $Cay(G,S)$. For a positive integer $r$, an $r$-perfect code or a perfect code of radius $r$ of $G$ under the metric $d_\Gamma$ is a subset $C$ of $G$ such that $G=\cup_{c\in C} B_r^\Gamma(c)$ and $B_r^\Gamma(c)\cap B_r^\Gamma(c')=\varnothing$ for any two distinct $c,c'\in C$. By a similar argument as the proof of Proposition \ref{rel}, it can be seen that if $C$ is an $r$-perfect code, then $\widehat{(S^r \cup \{1\})} \widehat{C} =\widehat{G}$. We note that according to Remark \ref{ballsize} $C$ is an $r$-perfect code if and only if  $|C||S^r\cup \{1\}|=|G|$. 

Let $\rho$ be any (complex) representation of a finite group $G$ of dimension $k$ for some positive integer $k$, i.e., any group homomorphism from $G$ to the general linear group $GL_k(\mathbb{C})$ of  $k\times k$ invertible matrices over $\mathbb{C}$. Then by the universal property of $\mathbb{C}[G]$, $\rho$ can be extended to an algebra homomorphism $\hat{\rho}$ from $\mathbb{C}[G]$ to the algebra $Mat_k(\mathbb{C})$ of $k\times k$ matrices over $\mathbb{C}$ such that
$g^{\hat{\rho}}=g^\rho$ for all $g\in G$. Thus the image of $\widehat{X}$ for any non-empty subset $X$ of $G$ under $\hat{\rho}$ is  the element $\sum_{x\in X}x^\rho$ of $Mat_k(\mathbb{C})$. 
In particular  by applying $\hat{\rho}$ on the equality \ref{groupring}, we obtain 	  
\begin{equation}\label{groupringmain}
	\big(\sum_{s\in S^r \cup\{1\}} s^\rho\big) \big(\sum_{c\in C}c^\rho\big)=\sum_{g\in G}g^\rho- \sum_{y\in Y }y^\rho, 
\end{equation} where
	the latter equality is between elements of $Mat_k(\mathbb{C})$. 

\begin{rem}\label{action}
Given a group $G$ and a non-empty set $X$, recall that we say $G$ acts on $X$ (from the right) if there exists a function $X\times G\rightarrow X$ denoted by $(x,g)\mapsto x^g$ for all $(x,g)\in X\times G$ if $(x^g)^h=x^{gh}$ and $x^1=x$ for all $x\in X$ and all $g,h\in G$, where $1$ denotes the identity element of $G$.  For any $x\in G$ The set 
$Stab_G(x):=\{g\in G\;|\; x^g=x\}$ is called the stabilizer of $x$ in $G$ which is a subgroup of $G$. If the action is transitive (i.e., for any two elements $x,y \in X$, there exists $g\in G$ such that $x^g=y$), all stabilizers are conjugate under the elements of $G$, more precisely $Stab_G(x)^g=Stab_G(y)$ whenever $x^g=y$, where $Stab_G(x)^g=g^{-1}Stab_G(x)g$. \\
Now suppose that $G$ acts on $X$ and $|X|=k$ is finite.  Fix an arbitrary ordering on the elements of $X$ so that $x_i<x_j$ whenever $i<j$ for distinct elements $x_i,x_j\in X$. Denote by $\rho_{X}^G$ the map from $G$ to $GL_k(\mathbb{Z})$ defined by $g\mapsto P_g$, where
$P_g$ is the $|X|\times |X|$ matrix whose $(i,j)$ entry is $1$ if $x_i^g=x_j$ and $0$ otherwise. 
Note that the definitions of $\rho_{X}^G$ depends on the choice of the ordering on $X$, however any two such representations of $G$ are conjugate by a permutation matrix.
\end{rem}

Let $H$ be a subgroup of a finite group $G$ and  $X$ be the set of  right cosets of $H$ in $G$, i.e., $X:=\{Hg\,|\, g\in G\}$. Then $G$ acts transitively on $X$ via $(Hg,g_0)\longrightarrow Hgg_0$. We note that, it is known that $X$  partitions  $G$, i.e., $G=\cup_{x\in X}{x}$ and $x\cap x'=\varnothing$ for all distinct elements $x$ and $x'$ of $X$, and $|X|=|G|/|H|$.
\begin{lem}\label{coset1}
Let $H$ be a subgroup of a finite group $G$ and  $X=\{Ha_1,\ldots ,Ha_m\}$ be the set of  right cosets of $H$ in $G$. If $\mathcal{Y}\subset G$, then by fixing the ordering  $Ha_i<Ha_j$ whenever $i<j$, the $(i,j)$ entry of  $\sum_{y\in\mathcal{Y}}y^{\rho_{X}^G}$  is $|\mathcal{Y}\cap {a_i}^{-1}Ha_j| $. 
\end{lem}
\begin{proof}
Clearly, for any $y\in \mathcal{Y}$, the $(i,j)$ entry of  $y^{\rho_{X}^G}=1$ if $Ha_iy=Ha_j$ and $0$ otherwise. So the $(i,j)$ entry of  $y^{\rho_{X}^G}=1$ if $a_i y {a_j}^{-1}\in H$ and therefore $y\in {a_i}^{-1}Ha_j$. Hence the $(i,j)$ entry of the $\sum_{y\in\mathcal{Y}}y^{\rho_{X}^G}$ is equal to $|\{y\in \mathcal{Y}\, |\,y\in {a_i}^{-1}Ha_j\}|$. This completes the proof. 
\end{proof}
\begin{thm}\label{intprogram}
Let $G$ be a finite group and $d_{\Gamma}$ be the metric induced by the metric $Cay(G,S)$. Also Let $C$ be a code in $G$ with $min\{d_{\Gamma}(c,c')\,|\,c,c'\in C\}\geq 3$. If $H$ is a subgroup of $G$ and $\mathcal{Z}$ is the set of  right cosets of $H$ in $G$, then the optimal value of the objective function of the following integer programming gives an upper bound on $|C|$
\begin{align*}\label{max}
\text{Maximize}&\quad \sum_{i=1}^{|\mathcal{Z}|}{x_i},\\
\text{subject to}&\quad \widehat{(S\cup \{1\})^{\rho_{\mathcal{Z}}^G}}(x_1,\ldots ,x_{|\mathcal{Z}|})^T \leq |H| \mathbf{1},\\
& \quad x_i\in \mathbb{Z},\,\, x_i\geq 0, \,\, i\in\{1,\ldots,|\mathcal{Z}|\},
\end{align*}
where $\textbf{1}$ is a column vector of order $|\mathcal{Z}|\times 1$ whose  entries are equal to $1$. 
\end{thm}
\begin{proof}
By Proposition \ref{rel}, there exists $Y\subset G$ such that
\begin{equation}
	\big(\sum_{s\in S \cup\{1\}} s^{\rho_{\mathcal{Z}}^G}\big) \big(\sum_{c\in C}c^{\rho_{\mathcal{Z}}^G}   \big)=\sum_{g\in G}g^{\rho_{\mathcal{Z}}^G}- \sum_{y\in Y }y^{\rho_{\mathcal{Z}}^G}, 
	\end{equation}
	Suppose that $\mathcal{Z}=\{Ha_1,\ldots ,Ha_m\}$. Without loss of generality, we may assume that $a_1=1$. We fix the ordering $Ha_i<Ha_j$ whenever $i<j$.
	By Lemma \ref{coset1}, the $(i,j)$ entry of $\sum_{g\in G}g^{\rho_{\mathcal{Z}}^G}$ is equal to $|G \cap {a_i}^{-1}Ha_j|$ and  since  $ {a_i}^{-1}Ha_j\subseteq G$, the $(i,j)$ entry of $\sum_{g\in G}g^{\rho_{\mathcal{Z}}^G}$ is equal to $|{a_i}^{-1}Ha_j|=|H|$, for all $i,j\in \{1,\ldots ,|\mathcal{Z}|\}$. So if $B$ is a column of $\sum_{g\in G}g^{\rho_{\mathcal{Z}}^G} $, then $B=|H| \mathbf{1}$. Let $\mathcal{C}$ be the first column of $\sum_{c\in C}c^{\rho_{\mathcal{Z}}^G}$. Then  Lemma \ref{coset1} implies that for all $1\leq i\leq |\mathcal{Z}|$,  $i$-th row of $\mathcal{C}$, denoted by $c_i$, is equal to $|C\cap Ha_i|$. Since $C=C\cap G=\cup_{i=1}^{|\mathcal{Z}|} (C\cap Ha_i)$ and $(C\cap Ha_i)\cap (C\cap Ha_j)=\varnothing$ for all $i\neq j$, $\sum_{i=1}^{|\mathcal{Z}|}c_i=|C|$. We note that by Lemma \ref{coset1}, all entries of matrix $\widehat{F^{\rho_{\mathcal{Z}}^G}}$, $F\in \{C,G,Y,(S\cup \{1\})\}$, are integer and non-negative. Therefore $\mathcal{C}$ is an integer solution for the following system of inequalities 
	\[\widehat{(S\cup \{1\})^{\rho_{\mathcal{Z}}^G}}(x_1,\ldots ,x_{|\mathcal{Z}|})^T \leq |H| \mathbf{1}
	\]
	such that $\sum_{i=1}^{|\mathcal{Z}|}c_i=|C|$ and this completes the proof.
\end{proof}

\section{Results}
Let $G=S_n$ and $S=\{(i,i+1)\,|\,1\leq i\leq n-1\}$, then the metric induced by $Cay(G,S)$ on $S_n$ is the Kendall $\tau$- metric. In this section, by using the result in Section II, we improve the upper bound of $P(n,3)$ when  $n\in \{6,14,15\} $ or $n\geq 7$ is a prime number.

Usual traditional with well-developed candidate for $\mathcal{B}$ is the set of Young tabloids of a given shape which we are going to recall them \cite{JK}. 
\begin{defn}\label{young1}
     By a number  partition $\lambda$ of $n$ (with the length $m$) we mean an $m$-tuple $(\lambda_1,\dots,\lambda_m)$ of positive integers such that $\lambda_1\geq \cdots\geq\lambda_m$ and $n=\sum_{i=1}^m \lambda_i$. Given a number partition $\lambda$ of $n$, by a Young tabloid of shape $\lambda$ we mean an $m$-tuple $(\mathfrak{n}_1,\dots,\mathfrak{n}_m)$ of non-empty subsets of $[n]$  consisting a set partition  for $[n]$ with $|\mathfrak{n}_i|=\lambda_i$ for all $i=1,\dots,m$. We denote by $\mathcal{YT}_n(\lambda)$ the set of all Young tabloids of shape $\lambda$ of $n$. With a Young tabloid $\mathfrak{n}=(\mathfrak{n}_1,\dots,\mathfrak{n}_m)$ of shape $\lambda$, we associate a Young subgroup $S_{\mathfrak{n}}$ of $S_n$  by taking
 $S_\mathfrak{n}=S_{\mathfrak{n}_1}\times \cdots\times S_{\mathfrak{n}_m}$. 
\end{defn}
\begin{rem}\label{young}
 The action  of $S_n$ on $\mathcal{YT}_n(\lambda)$ is defined by $(\mathfrak{n}_1,\dots,\mathfrak{n}_m)^\sigma=(\mathfrak{n}_1^\sigma,\dots,\mathfrak{n}_m^\sigma)$ for all $\sigma \in S_n$. Fix an arbitrary ordering  of the elements of $\mathcal{YT}_n(\lambda)$.
The representation $\rho_{\mathcal{YT}_n(\lambda)}^{S_n}$ is equivalent to the representation $\rho_{X}^{S_n}$, where $X$ is the set of  right cosets of the Young subgroup $S_{\mathfrak{n}}$ in $S_n$ for some Young tabloid $\mathfrak{n}=(\mathfrak{n}_1,\dots,\mathfrak{n}_m)$ of shape $\lambda$. Note that if $\mathfrak{n}$ and $\mathfrak{n}'$ are two Young tabloids of the same shape $\lambda$, the corresponding Young subgroups are conjugate in $S_n$ and so the representations $\rho_{X}^{S_n}$ and $\rho_{X'}^{S_n}$, where $X$ and $X'$ are the set of  right cosets of  Young subgroups $S_{\mathfrak{n}}$ and $S_{\mathfrak{n}'}$ in $S_n$, respectively,  are equivalent (i.e.,  a matrix $U$ exists such that $U^{-1}\rho_{X}^{S_n}(\sigma)U=\rho_{X'}^{S_n}(\sigma)$ for all $\sigma \in S_n$) so that we use Young tabloid $[\{1,\ldots ,\lambda_1\},\{\lambda_1+1,\ldots ,\lambda_1+\lambda_2\},\ldots,\{n-\lambda_m+1,\ldots ,n\}]$ for considering Young subgroup corresponding to the partition $\lambda=(\lambda_1,\ldots ,\lambda_m)$,  as we are studying these representations up to equivalence.   
\end{rem}
\begin{lem}\label{matrix}
Let $H$ be  Young subgroup of $S_n$ corresponding to the partition $\lambda:=(n-1,1)$ and $X$ be the set of  right cosets of $H$ in $S_n$. If $S=\{(i,i+1)\,|\,1\leq i\leq n-1\}$,  then $ \widehat{(S\cup \{1\})^{\rho_{X}^{S_n}}} $ is a  conjugate by a permutation matrix of the following matrix
 $$\begin{pmatrix}
	n-1   &1    &0     &0     &\dots &0  \\
	1     &n-2  &1     &0     &\dots &0    \\
	0     &1    &n-2   &1     &0     &0    \\
	\vdots&\dots&\ddots&\ddots&\ddots&\vdots\\
	0     &0    &\dots &1     &n-2   &1     \\
	0     &0    &\dots &0     &1     &n-1
	\end{pmatrix}.$$ 
\end{lem}
\begin{proof}
Let $T:=S\cup \{1\}$. Without loss of generality we may assume that $\lambda$ is the partition $\{\{1\},\{2,\ldots ,n\}\}$ of $n$ and therefore $H=Stab_{S_n}(1)$. If $i\neq j$, then  $H(1,i)\cap H(1,j)=\varnothing$ since otherwise $(1,i,j)\in H$ that is a contradiction.  So we can let  $X=\{H(1,i)\,|\,1\leq i\leq n\}$, where we are using the convention $H(1,1):=H$. We Fixing the ordering $H(1,i)<H(1,j)$ if $i<j$. By Lemma \ref{coset1}, the $(i,j)$ entry of $ \widehat{T^{\rho_{X}^{S_n}}} $ is equal to $|T\cap (1,i)H(1,j)|$. If $i=j$, then Remark \ref{action} imply $ (1,i)H(1,i)=Stab_{S_n}(i) $ and hence $T\cap (1,i)H(1,i)=T\setminus \{(i-1,i),(i,i+1)\} $ if $2\leq i\leq n-1$, $T\cap (1,n)H(1,n)=T\setminus \{(n,n-1)\} $ and $T\cap H=T\setminus \{(1,2)\} $. Now suppose that $i\neq j$. Clearly $(1,i)(i,j)(1,j)=(i,j)$. Let $h\in H$. Then  $\sigma:=(1,i)h(1,j)=\pi(1,j,i)$, where $\pi=(1,i)h(1,i)\in Stab_{S_n}(i)$. Since $\pi(i)=i$, $\sigma(j)=i$ and therefore $\sigma$ is an transposition if and only if $h=(i,j)$. Hence, if $j=i+1$ and  $i-1$, then  $T\cap (1,i)H(1,j)$ is equal to  $ \{(i,i+1)\} $ and $ \{(i-1,i)\} $, respectively, and otherwise $T\cap (1,i)H(1,j)=\varnothing$. This completes the proof.
\end{proof}
\begin{thm}\label{systemineq}
	Let $p\geq 7$ be a prime number and consider the $p\times p$ matrix $$M=\begin{pmatrix}
	p-1   &1    &0     &0     &\dots &0  \\
	1     &p-2  &1     &0     &\dots &0    \\
	0     &1    &p-2   &1     &0     &0    \\
	\vdots&\dots&\ddots&\ddots&\ddots&\vdots\\
	0     &0    &\dots &1     &p-2   &1     \\
	0     &0    &\dots &0     &1     &p-1
	\end{pmatrix}.$$ Consider the system of inequalities $M{\mathbf x}\leq (p-1)! \mathbf{1}$ with $\mathbf{x}\geq \mathbf{0}$ and $x_i$ are integers. Let $x_{\max}:=\max\{x_i\;|\; i=1,\dots,p\}$. Then
\begin{enumerate}
\item $|\{i\in[p] \;|\; x_i\leq \frac{(p-1)!}{p}\}|\geq \lceil\frac{p}{3}\rceil$.
\item If $\sum_{i=1}^p x_i=(p-1)!-k$, then $|\{i \;|\; x_i=x_{\max}\}|\geq p-k-2$.
\item  $\sum_{i=1}^p x_i \leq (p-1)!-\lceil \frac{p}{3}\rceil+2$	
\end{enumerate} 
\end{thm}
\begin{proof}
Let $\mathcal{A}:=\{i\in[p] \;|\; x_i\leq \frac{(p-1)!}{p}\}$ and $\mathcal{B}:=\{i \;|\; x_i=x_{\max}\}$. Consider the partition $\{\{1,2\},\{3,4,5\},\{6,7,8\},\dots,\{p-2,p-1,p\}\}$ of $[p]$ if $p\equiv 2 \mod 3$ and the partition $\{\{1,2\},\{3,4,5\},\{6,7,8\},\dots,\{p-4,p-3,p-2\},\{p-1,p\}\}$ if $p\equiv 1 \mod 3$.  Each  member of partitions corresponds to an obvious inequality, e.g.
$\{1,2\}$ and $\{p-2,p-1,p\}$ are respectively corresponding to  $(p-1)x_1+x_2\leq (p-1)!$ and  $x_{p-2}+(p-2)x_{p-1}+x_p\leq (p-1)!$. Each  inequality corresponding to a member $P$ of partitions forces $x_i\leq (p-1)!/p$ for some $i\in P$, where $x_i=\min\{x_j\;|\; j\in P\}$. Since the size of both partitions is the same $\lceil \frac{p}{3}\rceil$, $|\mathcal{A}|\geq \lceil\frac{p}{3}\rceil$. 

Let $\ell\in [p]$ be such that $x_\ell=x_{\max}$. Thus $\sum_{i=1, i\neq \ell-1,\ell+1}^p (x_{\ell} -x_i)=x_{\ell-1}+(p-2)x_{\ell}+x_{\ell+1}-\sum_{i=1}^{p}x_i\leq (p-1)!-((p-1)!-k)$.
 Thus $\sum_{i=1, i\neq \ell-1,\ell+1}^p (x_\ell -x_i) \in \{0,1,\dots,k\}$. It follows that $|\{i \;|\; x_i< x_{\max}\}|\leq k+2$ and so  $|\mathcal{B}|\geq p-k-2$.

Let $\sum_{i=1}^p x_i=(p-1)!-k$ and suppose, for a contradiction, that $k<\lceil \frac{p}{3} \rceil -2$. So $|\mathcal{B}|\geq p-\lceil \frac{p}{3} \rceil+1$ and therefore
$$|\mathcal{A}\cap \mathcal{B}|\geq |\mathcal{A}|+|\mathcal{B}|-p\geq \lceil \frac{p}{3} \rceil+p-\lceil \frac{p}{3} \rceil+1-p\geq 1. $$ 

 Hence $\mathcal{A}\cap \mathcal{B}\neq \varnothing$ and  $x_{\max}\leq (p-1)!/p$. Since $p$ is  prime, by Wilson theorem \cite[P. 27]{wilson} $(p-1)!\equiv -1 \mod p$. Since $x_{\max}$ is integer, we have that $x_i\leq \frac{(p-1)!+1}{p}-1$ for all $i\in [p]$. Therefore $$\sum_{i=1}^px_i=(p-1)!-k\leq p(\frac{(p-1)!+1}{p}-1)=(p-1)!+1-p$$ and so $$p\leq k+1<\lceil \frac{p}{3} \rceil -1,$$ which is a contradiction. So we must have $k\geq \lceil \frac{p}{3} \rceil -2$. This completes the proof.
\end{proof}
\begin{cor}
For all primes $p\geq 11$,   $P(p,3)\leq (p-1)!-\lceil \frac{p}{3}\rceil+2\leq (p-1)!-2$.
\end{cor}
\begin{proof}
The result follows from Theorems \ref{intprogram}, \ref{systemineq} and Lemma \ref{matrix}.
\end{proof}
\begin{thm}
If $n$ is equal to $6$, $7$, $11$, $13$ and $17$, then $P(n,3)$ is less than or equal to $116$, $716$, $10!-10$, $12!-12$  and $16!-5$, respectively.
\end{thm}
\begin{proof}
Let $S:=\{(i,i+1)\,|\,1\leq i\leq n-1\}$. In view of Theorem \ref{intprogram}, we have used a CPLEX software \cite{cplex} to determine the upper bound for $P(n,3)$ obtained from solving the integer programming corresponding to the subgroup $H$ of $S_n$, where $H$ is the Young subgroup corresponding to the partition $(2,2,2)$, when $n=6$, $(5,1,1)$, when $n=7$, $(9,2)$, when $n=11$, $(11,2)$, when $n=13$  and $(16,1)$, when $n=17$.
\end{proof}
\begin{thm}
There are no $1$-perfect codes under the Kendall $\tau$-distance in $S_n$ when $n\in\{14,15\}$.
\end{thm}
\begin{proof}
Let $S=\{(i,i+1)\,|\,1\leq i\leq n-1\}$ and $T:=S\cup \{1\}$. We are using  techniques in \cite{white} for proving this theorem. By \cite[Theorem 2.2]{white}, if $S_n$ contains a subgroup $H$ such that $n\nmid |H|$ and $\widehat{(T)^{\rho_{X}^{S_{n}}}}$ is invertible, where $X$ is the set of  right cosets of $H$ in $S_n$, then  $S_n$ contains no  $1$-perfect codes under the kendall $\tau$-distance. In the case $n=14$, we consider   Young subgroup $H$  corresponding to the partition $(6,6,2)$. By a software check the matrix $\widehat{(T)^{\rho_{X}^{S_{14}}}}$ is invertible.\\ 
In the case $n=15$, we consider   Young subgroup $H$  corresponding to the partition
 $(4,4,4,3)$  in $S_{15}$.   By \cite[Corollary 2.2.22]{JK}, if for all \\ $\lambda\in\{(15),(14,1),(13,2),(13,1,1),(12,3),(12,2,1),(11,4),(8,7),\\(10,5),(11,2,2),(9,6),(11,3,1),(7,7,1),(5,5,5),(10,4,1),\\(10,3,2),(9,5,1),(8,6,1),(9,3,3),(9,2,2,2),(6,6,3),(9,4,2),\\(4,4,4,3),(7,6,2),(7,4,4),(6,5,4),(8,5,2),(8,4,3),(7,5,3),\\(6,3,3,3),(8,3,2,2),(5,4,4,2),(7,3,3,2),(5,5,3,2),(6,5,2,2),\\(7,4,2,2),(6,4,3,2)\}$, $\widehat{T^{\rho_{\lambda}}}$ is invertible, where $\rho_{\lambda}$ is the irreducible representation of $S_{15}$ corresponding to $\lambda$, then $\widehat{T^{\rho_{X}^{S_{15}}}}$ is invertible. By software check all these matrices are invertible. This completes the proof.
\end{proof}
\begin{con}
If $H$ is the Young subgroup corresponding to the partition $(p-1,p-1,2)$ of $S_{2p}$, where $p\geq 3$ is a prime number,  and $X$ is the set of  right cosets of $H$ in $S_{2p}$, then $\widehat{(S\cup \{1\})^{\rho_{X}^{S_{2p}}}}$ is invertible. In particular,  there is no $1$-perfect permutation code of length $2p$ with respect to the Kendall $\tau$-distance.
\end{con}
\section{Conclusion}
Due to the applications of PCs under the Kendall $\tau$-distance in flash memories, they have attracted the attention of many researchers. In this paper, we consider the upper bound of  the size of the largest PC with minimum Kendall
$\tau$-distance 3. Using Group theory, corresponding to any non-trivial subgroup of $S_n$, we formulate  an integer programming  where the optimal
value of the objective function  gives an upper bound on $P(n,3)$.  After that, by solving the integer programming corresponding to some subgroups of $S_n$, when $n\geq 7$ is a prime number or $n\in\{6,14,15\}$, we can improve the upper bound on $P(n,3)$.

\end{document}